\documentclass[12pt]{amsart}
\usepackage{amssymb}
\usepackage{amsmath,latexsym}
\usepackage{amsfonts}
\usepackage{amsthm}
\usepackage{eucal}
\allowdisplaybreaks[3]
\numberwithin{equation}{section}

\newtheorem{thm}{Theorem}[section]
\newtheorem{defn}[thm]{Definition}
\newtheorem{lem}[thm]{Lemma}
\newtheorem{prop}[thm]{Proposition}
\newtheorem{cor}[thm]{Corollary}

\newtheorem{remark}[thm]{Remark}

\begin{document}

\title[]{Stability of Riemannian manifolds with Killing spinors}

\author{Changliang Wang}
\address{Department of Mathematics and Statistics, McMaster University, Hamilton, Ontario, Canada}
\email{wangc114@math.mcmaster.ca}
\date{}

\maketitle

\begin{abstract}
Riemannian manifolds with non-zero Killing spinors are Einstein manifolds. Klaus Kr\"{o}ncke proved that all complete Riemannian manifolds with imaginary Killing spinors are (linearly) strictly stable in \cite{Kro15}. In this paper, we obtain a new proof for this stability result by using a Bochner type formula in \cite{DWW05} and \cite{Wan91}. Moreover, existence of real Killing spinors is closely related to the Sasaki-Einstein structure. A regular Sasaki-Einstein manifold is essentially the total space of a certain principal $S^{1}$-bundle over a K\"{a}hler-Einstein manifold. We prove that if the base space is a product of two K\"{a}hler-Einstein manifolds then the regular Sasaki-Einstein manifold is unstable. This provides us many new examples of unstable manifolds with real Killing spinors.
\end{abstract}

\section{Introduction}
\noindent Einstein metrics naturally come out of some variational problems. For example Einstein metrics on a compact manifold $M$ are the critical points of the total scalar curvature functional with the fixed volume 1. Then the stability problem naturally comes up when we consider the second variation of the total scalar curvature functional with the fixed volume 1 at an Einstein metric $g$. The second variation formula is given by $-\frac{1}{2}\langle\nabla^{*}\nabla h-2\mathring{R}h, h\rangle_{L^{2}(M)}$, when restricted in traceless transverse direction, i.e. $h\in C^{\infty}(M, S^{2}(M))$ satisfying $tr_{g}h=0$ and $\delta_{g}h=0$, where $S^{2}(M)$ is the bundle of symmetric 2-tensors, $(\mathring{R}h)_{ij}=R_{ikjl}h^{kl}$, and $\delta_{g}h$ is the divergence of $h$. An Einstein manifold $(M^{n}, g)$ is said to be stable if $\langle \nabla^{*}\nabla h-2\mathring{R}h, h\rangle_{L^{2}(M)}\geq0$ for all traceless transverse symmetric 2-tensors $h$, and otherwise, $(M^{n}, g)$ is unstable. $(M^{n}, g)$ is said to be strictly stable if $\langle \nabla^{*}\nabla h-2\mathring{R}h, h\rangle_{L^{2}(M)}\geq c\langle h,h\rangle_{L^{2}(M)}$ for some constant $c>0$. The operator $\nabla^{*}\nabla-2\mathring{R}$ acting on symmetric 2-tensors in $C^{\infty}(M, S^{2}(M))$ is called the Einstein operator. If the manifold is non-compact, we only consider compactly supported symmetric 2-tensors $h$. This stability problem has been extensively studied, see e.g. \cite{Koi78}, \cite{Koi79}, \cite{Koi80}, \cite{DWW05}, \cite{DWW07}, \cite{Kro15}, and also see the book \cite{Bes87} for an introduction of this stability problem and some discussions of many interesting results on this problem.

The stability problem of Einstein metrics was also similarly studied with respect to variation formulae of Perelman's $\nu$-entropy (see, e.g. \cite{Per02} and \cite{CZ12}) for Einstein metrics with positive Ricci curvature, and also variation formulae of $\nu_{+}$-entropy (see, e.g. \cite{FIN05} and \cite{Zhu11}) for Einstein metrics with negative Ricci curvature. For example, H-D. Cao and C. He studied stability of Einstein metrics with respect to $\nu$-entropy on symmetric spaces of compact type in \cite{CH13}.

In this paper, we will study the stability of complete Riemannian manifolds with non-zero Killing spinors, which then are Einstein manifolds. These manifolds are important in both mathematics and physics. Th. Friedrich initiated the mathematical investigation of Killing spinors in \cite{Fri80}. And then complete Riemannian manifolds with Killing spinors were classified in \cite{Bar93}, \cite{Bau89_1}, \cite{Bau89}, \cite{FK89}, and \cite{FK90} (also see the book \cite{BFGK91}). Riemannian manifolds with real and imaginary Killing spinors have several very distinct properties. For example, Riemannian manifolds with non-zero real Killing spinors are compact. On the other hand, Riemannian manifolds with non-zero imaginary Killing spinors are non-compact (see \cite{CGLS86} and \cite{Bau89}). So we study the stability of these two kinds of manifolds separately.

If we allow a Killing constant to be zero, then parallel spinors can be viewed as Killing spinors with zero Killing constant. And in particular, Riemannian manifolds with non-zero parallel spinors are Ricci-flat, i.e. Ricci curvature is zero. X. Dai, X. Wang, and G. Wei proved that manifolds with non-zero parallel spinors are stable in \cite{DWW05} by deriving a Bochner type formula, and rediscovery a result in \cite{Wan91}.

Then it is very natural to ask whether we can estimate the Einstein operator on manifolds with Killing spinors and further conclude some stability results for these manifolds by extending X. Dai, X. Wang and G. Wei's Bochner type argument to Killing spinor case because Killing spinors give us a similar Bochner type formula as parallel spinors. We will answer this question for imaginary and real Killing spinors separately.

Recall in \cite{Bau89} an imaginary Killing spinor $\sigma$ is called to be of type \uppercase\expandafter{\romannumeral1} if there exists a vector field $X$ such that $X\cdot\sigma=\sqrt{-1}\sigma$, where $`` \cdot "$ denotes the Clifford multiplication, and otherwise, $\sigma$ is of type \uppercase\expandafter{\romannumeral2}. As mentioned in \cite{Kro15}, non-constant length functions of imaginary Killing spinors will cause some issues for extending the Bochner type argument to the imaginary Killing spinor case. Here, we overcome this difficulty and obtain the following estimate for Einstein operator on complete Riemannian manifolds with imaginary Killing spinors of type \uppercase\expandafter{\romannumeral1} by using a Bochner type formula in \cite{DWW05} and \cite{Wan91}. Later on, we will see that type \uppercase\expandafter{\romannumeral1} imaginary Killing spinors are the only interesting ones for us, since the stability of complete manifolds with type \uppercase\expandafter{\romannumeral2} imaginary Killing spinors has been fully understood.

\begin{thm}\label{EinsteinOperatorEstimateType1}
Let $(M^{n},g)$ be a complete Riemannian manifold with a non-zero imaginary Killing spinor of type \uppercase\expandafter{\romannumeral1} with the imaginary Killing constant $\mu$. We have
\begin{equation}
\int_{M}\langle\nabla^{*}\nabla h-2\mathring{R}h, h\rangle dvol_{g}\geq-[n(n-2)-4]\mu^{2}\int_{M}\langle h,h\rangle dvol_{g}.
\end{equation}
for all compactly supported traceless transverse symmetric 2-tensor $h$.
\end{thm}
\begin{cor}\label{ImaginaryStability}
Complete Riemannian manifolds with non-zero imaginary Killing spinors are strictly stable.
\end{cor}

H. Baum proved that $n$-dimensional complete Riemannian manifolds with imaginary Killing spinors of type \uppercase\expandafter{\romannumeral2} with Killing constant $\sqrt{-1}\nu$ are isometric to the $n$-dimensional hyperbolic space $H^{n}_{-4\nu^{2}}$ with constant sectional curvature $-4\nu^{2}$. N. Koiso proved that Einstein manifolds with negative sectional curvature, in particular, hyperbolic spaces, are stable in \cite{Koi79} (also see \cite{Bes87}). Indeed, by the first inequality in 12.70 in \cite{Bes87}, one can see that $\langle\nabla^{*}\nabla h-2\mathring{R}h, h\rangle_{L^{2}}\geq4(n-2)\nu^{2}\langle h, h\rangle_{L^{2}}$ for all compactly supported traceless transverse 2-tensors $h$ on the hyperbolic space $H^{n}_{-4\nu^{2}}$. Therefore, we focus on Riemannian manifolds with imaginary Killing spinors of type \uppercase\expandafter{\romannumeral1} and by applying Theorem $\ref{EinsteinOperatorEstimateType1}$ we obtain a new prove for Corollary $\ref{ImaginaryStability}$, which has been proved in \cite{Kro15}.

In the case of real Killing spinors, which have constant length functions, by doing an integration by parts, the Bochner type formula in \cite{DWW05} and \cite{Wan91} gives us a lower bound $-(n-1)^{2}\mu^{2}$ for the Einstein operator on $n$-dimensional manifolds with real Killing spinors with Killing constant $\mu$. In \cite{GHP03}, they used essentially the same argument to obtain a lower bound $-(n^{2}-10n+9)\mu^{2}$ for the Lichnerowicz Laplacian on $n$-dimensional manifolds with real Killing spinors with the Killing constant $\mu$. On an Einstein manifold, the Einstein operator is just a constant shift of the Lichnerowicz Laplacian, and these two estimates are equivalent.

Unlike the case of imaginary Killing spinors, from this estimate we cannot conclude a general stability result. Actually, we have both stable and unstable examples: standard spheres are stable Riemannian manifolds with real Killing spinors; the Jensen's sphere (also called the squashed sphere) is an unstable Riemannian manifold with a real Killing spinor. We refer to \cite{ADP83}, \cite{Bar93}, \cite{Bes87}, \cite{Jen73}, and \cite{Spa11} for the Jensen's sphere.

Riemannian manifolds with non-zero real Killing spinors are either standard spheres in even dimensions, except in 6 dimension, or Sasaki-Einstein in odd dimensions, except in 7 dimension. We know that standard spheres are strictly stable. Therefore, we focus on the stability of Sasaki-Einstein manifolds, especially regular Sasaki-Einstein manifolds, which then are the total spaces of principal $S^{1}$-bundles over K\"{a}hler-Einstein manifolds. We will use this structure property to study the stability of regular Sasaki-Einstein manifolds, and we obtain many new unstable examples of manifolds with real Killing spinors.

Let $\pi: (M^{2p+1}, g)\rightarrow (B^{2p}, G, J)$ be a principal $S^{1}$-bundle with a connection $\eta$, where $(M^{2p+1}, g)$ is regular Sasaki-Einstein, $(B^{2p}, G, J)$ is K\"{a}hler-Einstein, and $\pi$ is a Riemannian submersion. Here $G$ is the K\"{a}hler metric on $B^{2p}$, and $J$ is the almost complex structure on $B^{2p}$. We denote $\tilde{h}=\pi^{*}h$, for all symmetric 2-tensors $h\in C^{\infty}(B, S^{2}(B))$. Then we have the following relationship between Einstein operators on the base Einstein manifold $B$ and the total Einstein manifold $M$.
\begin{thm}\label{EinsteinOperatorRelation1}
\begin{equation}
\langle(\nabla^{g})^{*}\nabla^{g}\tilde{h}-2\mathring{R}^{g}\tilde{h}, \tilde{h}\rangle=(\langle(\nabla^{G})^{*}\nabla^{G}h-2\mathring{R}^{G}h, h\rangle+4\langle h, h\rangle+4\langle h\circ J, h\rangle)\circ\pi,
\end{equation}
where $h\circ J\in C^{\infty}(B, S^{2}(B))$ with $h\circ J(X, Y)=h(JX, JY)$.
\end{thm}
\begin{cor}
If there exists a traceless transverse symmetric 2-tensor $h\in C^{\infty}(B, S^{2}(B))$ such that $\int_{B}(\langle(\nabla^{G})^{*}\nabla^{G}h-2\mathring{R}^{G}h, h\rangle dvol_{G}<-8\int_{B}\langle h,h\rangle dvol_{G}$, then $(M^{2p+1}, g)$ is unstable.
\end{cor}

Moreover, we know the Riemannian product of two Einstein manifolds with the same Einstein constant is an unstable Einstein manifold with a canonical unstable direction. By applying Theorem $\ref{EinsteinOperatorRelation1}$, we obtain that the lift of this canonical unstable direction is an unstable direction on the total space if the base is the product of two K\"{a}hler-Einstein manifolds, and therefore we have the following corollary.
\begin{cor}\label{UnstableRealKillingSpinors}
If the base space $(B^{2p}, g)$ is a product of two K\"{a}hler-Einstein manifolds, then $(M^{2p+1}, g)$ is unstable.
\end{cor}

Corollary $\ref{UnstableRealKillingSpinors}$ provides us many new examples of unstable manifolds with real Killing spinors. Before this, the Jensen's sphere is the only known example of unstable manifolds with real Killing spinors. However, the Jensen's sphere is very speical. It is the only known manifold with exactly one linearly independent real Killing spinor, and therefore it is not Sasaki-Einstein. Corollary $\ref{UnstableRealKillingSpinors}$ shows the existence of more generic unstable manifolds with real Killing spinors.

The paper is organized as follows. In Section 2, we will review classification results of Riemannian manifolds with Killing spinors and some properties of imaginary Killing spinors, which will be used in Section 4. In Section 3, we will present a proof of the Bochner type formula from Killing spinors in \cite{DWW05} and \cite{Wan91}. In Section 4, we will prove Theorem $\ref{EinsteinOperatorEstimateType1}$ and Corollary $\ref{ImaginaryStability}$. In Section 5, we will briefly discuss stability of manifolds with real Killing spinors by using the Bochner type formula. In Section 6, we will give a brief introduction for Sasaki-Einstein manifolds, and prove Theorem $\ref{EinsteinOperatorRelation1}$ and its corollaries.


\section{Riemannian manifolds with imaginary Killing spinors}
\noindent In this section, we review classification results of Riemannian manifolds with Killing spinors and some properties of Killing spinors. We will mainly focus on complete Riemannian manifolds with imaginary Killing spinors studied in \cite{Bau89_1} and \cite{Bau89}, because Baum's results about the structure of complete Riemannian manifolds with imaginary Killing spinors play a very important role in our estimate of the Einstein operator on these manifolds.

Let $(M^{n},g)$ be a Riemannian manifold with a non-zero Killing spinor $\sigma$ with the Killing constant $\mu\neq0$, i.e.
\begin{equation}\label{KillingSpinorEquation}
\nabla^{S}_{X}\sigma=\mu X\cdot\sigma,
\end{equation}
for any vector field $X$, where $\nabla^{S}$ denotes the canonical connection on the spinor bundle induced by the Levi-Civita connection on the tangent bundle $TM$, and `` $\cdot$ " denotes the Clifford multiplication. Then the Riemannian manifold $(M^{n},g)$ is an Einstein manifold with scalar curvature $R=4n(n-1)\mu^{2}$ (see, e.g. \cite{Fri00}). Because the scalar curvature is real, $\mu$ can only be real or purely imaginary. A non-zero Killing spinor is said to be imaginary (resp. real) if its Killing constant is imaginary (resp. real). For details about spin geometry, we refer to \cite{Fri00} and \cite{LM89}

Let us first recall two differences between manifolds with real Killing spinors and manifolds with imaginary Killing spinors pointed out in \cite{Bau89} (also see \cite{CGLS86}):
\begin{enumerate}
\item Let $(M^{n},g)$ be a complete Riamnnian manifold with a Killing spinor $\sigma$. If $\sigma$ is real (with a non-zero real Killing constant), then $M^{n}$ is compact. If $\sigma$ is imaginary, then $M^{n}$ is non-compact.
\item Let $f(x):=\langle \sigma(x),\sigma(x)\rangle_{\mathcal{S}_{x}}$ denote  the length function of a non-zero Killing spnior $\sigma$. If $\sigma$ is real, then $f$ is constant. If $\sigma$ is imaginary, then $f$ is a non-constant and nowhere vanishing function.
\end{enumerate}

As pointed out by Klaus Kr\"{o}ncke in \cite{Kro15}, the fact that the length function $f$ of an imaginary Killing spinor is not constant will cause some issues when we use the Bochner type argument in \cite{DWW05} to estimate the Einstein operator on a Riemannian manifold with imaginary Killing spinors. In order to deal with the issues, we investigate the length function $f$ more carefully, and we recall some properties of the length function $f$ proved in \cite{Bau89}. Let $(M^{n},g)$ be a complete Riemannian manifold with an imaginary Kiling spinor $\sigma$ with Killing constant $\mu=\sqrt{-1}\nu$.
\begin{lem}[\cite{Bau89}]
\begin{enumerate}
\item The function
      \begin{equation}\label{Equaiton1ForLengthFunction}
      q_{\sigma}(x):=f^{2}(x)-\frac{1}{4\nu^{2}}|\nabla f(x)|^{2}
      \end{equation}
      is constant on $M^{n}$.
\item Let $\{e_{1},\cdots,e_{n}\}$ be a local orthonormal frame of $TM$ around $x$. The we have
      \begin{equation}\label{OrthogonalityOfSpinors}
      Re\langle e_{i}\cdot\sigma(x),e_{j}\cdot\sigma(x)\rangle=\delta_{ij}f(x),
      \end{equation}
      where $Re$ means taking the real part.
\item Let $dist$ denote the distance in $\mathcal{S}_{x}$ with respect to the real scalar product $Re\langle,\rangle_{\mathcal{S}_{x}}$. Then
      \begin{equation}\label{Equation2ForLengthFunction}
      q_{\sigma}=f(x)\cdot dist^{2}(V_{\sigma},\sqrt{-1}\sigma(x))\geq0,
      \end{equation}
      where $V_{\sigma}(x)=\{X\cdot\sigma(x)| \ \ X\in T_{x}M\}\subset \mathcal{S}_{x}$.
\end{enumerate}
\end{lem}

As in \cite{Bau89}, a Killing spinor $\sigma$ is of type \uppercase\expandafter{\romannumeral1} if $q_{\sigma}=0$ and a Killing spinor is of type \uppercase\expandafter{\romannumeral2} if $q_{\sigma}>0$. By $(\ref{Equation2ForLengthFunction})$, this is equivalent to the simple characteristic of Killing spinors of type \uppercase\expandafter{\romannumeral1} and \uppercase\expandafter{\romannumeral2} mentioned in Introduction. H. Baum has the following classification results for complete Riemannian manifold with imaginary Killing spinors.
\begin{thm}[\cite{Bau89}]\label{ClassificationOfType2}
Let $(M^{n},g)$ be a complete connected Riemannian manifold with an imaginary Killing spinor of type \uppercase\expandafter{\romannumeral2} with the Killing constant $\sqrt{-1}\nu$. Then $(M^{n},g)$ is isometric to the hyperbolic space $H^{n}_{-4\nu^{2}}$ with the constant sectional curvature $-4\nu^{2}$.
\end{thm}
\begin{thm}[\cite{Bau89_1}, \cite{Bau89}]\label{ClassificationOfType1}
Let $(M^{n},g)$ be a complete connected Riemannian manifold with an imaginary Killing spinor of type \uppercase\expandafter{\romannumeral1} with the Killing constant $\sqrt{-1}\nu$. Then $(M^{n},g)$ is isometric to a warped product $(F^{n-1}\times\mathbb{R}, e^{-4\nu t}h+dt^{2})$, where $(F^{n-1},h)$ is a complete Riemannian manifold with a non-zero parallel spinor.

Conversely, let $(F^{n-1},h)$ be a complete Riemannian manifold with non-zero parallel spinors, then the warped product $(M^{n},g):=(F^{n-1}\times\mathbb{R}, e^{-4\nu t}h+dt^{2})$ is a complete Riemannian manifold with imaginary Killing spinors of type \uppercase\expandafter{\romannumeral1}.
\end{thm}

Recall how to construct a Killing spinor of type \uppercase\expandafter{\romannumeral1} on $(F^{n-1}\times\mathbb{R}, e^{-4\nu t}h+dt^{2})$ from a parallel spinor on $(F^{n-1},h)$. When $n-1$ is even, the spinor bundle over the warped product $(F^{n-1}\times\mathbb{R}, e^{-4\nu t}h+dt^{2})$ is isometric to the tensor product of the spinor bundle over $(F^{n-1},h)$ and the spinor bundle over $(\mathbb{R},dt^{2})$. When $n-1$ is odd, the spinor bundle over $(F^{n-1}\times\mathbb{R}, e^{-4\nu t}h+dt^{2})$ is isometric to the direct sum of two copies of the tensor product of the spinor bundle over $(F^{n-1}, h)$ and the spinor bundle over $(\mathbb{R}, dt^{2})$. The spinor bundle over $(\mathbb{R},dt^{2})$ is a trivial 1-dimensional complex vector bundle. We will use the same notation to denote two isometric spinors.
\begin{enumerate}
\item If $n-1$ is even, and parallel spinor on $F^{n-1}$ is $\psi=(\psi^{+},\psi^{-})$, where the decomposition is the $\sqrt{-1}$ and $-\sqrt{-1}$ eigenspaces decomposition for the action of the complex volume $\omega_{\mathbb{C}}=(\sqrt{-1})^{\frac{n}{2}}e_{1}\cdots e_{n-1}$ on the spinor bundle on $F^{n-1}$, then we can take
    \begin{equation}\label{KillingSpinor1}
    \sigma=e^{-\nu t}\psi^{+}\otimes1
    \end{equation}
    as an imaginary Killing spinor of type \uppercase\expandafter{\romannumeral1} on the warped product manifold.
\item If $n-1$ is odd, and parallel spinor on $F^{n-1}$ is $\psi$, then we can take
    \begin{equation}\label{KillingSpinor2}
    \sigma=e^{-\nu t}(\psi\otimes1,\hat{\psi}\otimes1)
    \end{equation}
    as a Killing spinor of type \uppercase\expandafter{\romannumeral1} on the warped product manifold, where $`` \ \ \hat{} \ \ "$ denotes the isomorphism between two spin representations coming from projections to the first and the second components of $Cl(n-1)\otimes\mathbb{C}=End(\mathbb{C}^{\frac{n-2}{2}})\oplus End(\mathbb{C}^{\frac{n-2}{2}})$.
\end{enumerate}

Because the length of a parallel spinor is constant, we can always normalize the parallel spinor $\psi$ on $F$ so that for the Killing spinor $\sigma$ in $(\ref{KillingSpinor1})$ and $(\ref{KillingSpinor2})$ we have
$$\langle\sigma,\sigma\rangle=e^{-2\nu t}.$$
Thus for the Killing spinor obtained above we have the length function
\begin{equation}\label{LengthFunction}
f=e^{-2\nu t}
\end{equation}
only depending on the $t$ variable on $\mathbb{R}$ factor. We can also see that $q_{\sigma}=0$. Moreover, we can see that the action of the vector field $\frac{\partial}{\partial t}$ on the Killing spinor $\sigma$ is given by
\begin{equation}\label{tDirectionActionOnKillingSpinor}
(\frac{\partial}{\partial t})\cdot\sigma=\sqrt{-1}\sigma.
\end{equation}


\section{Bochner type formula}
\noindent In this section, we recall a Bochner type formula coming from Killing spinors in \cite{DWW05} and \cite{Wan91} and present a proof.

Let $(M^{n},g)$ be a Riemannian spin manifold with spinor bundle $\mathcal{S}\rightarrow M$. The curvature of a connection $\nabla$ on a vector bundle $E\rightarrow M$ is defined as
\begin{equation}\label{DefOfCurvature}
R_{XY}\sigma=-\nabla_{X}\nabla_{Y}\sigma+\nabla_{Y}\nabla_{X}\sigma+\nabla_{[X,Y]}\sigma,
\end{equation}
for a section $\sigma\in C^{\infty}(M, E)$ and vector field $X,Y\in C^{\infty}(M, TM)$. Let $R^{S}$ be the curvature of $\nabla^{S}$ on the spinor bundle. Let $\{e_{1},\cdots,e_{n}\}$ be a local orthonormal frame of the tangent bundle and $\{e^{1},\cdots,e^{n}\}$ be its dual frame. We have
\begin{equation}\label{CurvatureOnSpinorBundle}
R^{S}_{XY}\sigma=\frac{1}{4}R(X,Y,e_{i},e_{j})e_{i}e_{j}\cdot\sigma,
\end{equation}
for any spinor $\sigma$. If there exists a Killing spinor $\sigma$ with Killing constant $\mu$, the Ricci curvature tensor satisfies
\begin{equation}\label{RicciCurvatureWithKillingSpinor}
R_{ij}=4\mu^{2}(n-1)g_{ij},
\end{equation}
(see, e.g. \cite{Fri00}). As in \cite{DWW05}, we define a linear map $\Phi: S^{2}(M)\rightarrow \mathcal{S}\otimes T^{*}M$ as
\begin{equation}\label{DefinitionOfPhi}
\Phi(h)=h_{ij}e_{i}\cdot\sigma\otimes e^{j}.
\end{equation}

\begin{prop}[\cite{DWW05}, \cite{Wan91}]\label{BochnerTypeForumula}
Let $D$ be the twisted Dirac operator acting on $\mathcal{S}\otimes T^{*}M$, and $h$ be a symmetric 2-tensor on $M$. Then
\begin{equation}\label{Bochner}
\begin{aligned}
D^{*}D\Phi(h)= &\Phi((\nabla^{*}\nabla-2\mathring{R})h)+n(n-2)\mu^{2}\Phi(h)+2\mu D\Phi(h)\\
&+4\mu^{2}(trh)e_{j}\cdot\sigma\otimes e^{j}-4\mu(\delta h)_{j}\cdot\sigma\otimes e^{j}.
\end{aligned}
\end{equation}
\end{prop}
\begin{proof}
Fix a point $x\in M$, choose a local orthonormal frame $\{e_{1},\cdots,e_{n}\}$ around $x$ such that $\nabla e_{i}=0$ at x. Then, at x,
\begin{equation}\label{DDPhi(h)}
\begin{aligned}
D^{*}D\Phi(h)
&=\nabla_{e_{k}}\nabla_{e_{l}}h_{ij}e_{k}e_{l}e_{i}\cdot\sigma\otimes e^{j}
  +\nabla_{e_{l}}h_{ij}e_{k}e_{l}e_{i}\cdot\nabla^{S}_{e_{k}}\sigma\otimes e^{j}\\
&\ \ \ +\nabla_{e_{k}}h_{ij}e_{k}e_{l}e_{i}\cdot\nabla^{S}_{e_{l}}\sigma\otimes e^{j}
  +h_{ij}e_{k}e_{l}e_{i}\cdot\nabla^{S}_{e_{k}}\nabla^{S}_{e_{l}}\sigma\otimes e^{j}\\
&=\nabla_{e_{k}}\nabla_{e_{l}}h_{ij}e_{k}e_{l}e_{i}\cdot\sigma\otimes e^{j}
  +\nabla_{e_{l}}h_{ij}(e_{k}e_{l}+e_{l}e_{k})e_{i}\cdot\nabla^{S}_{e_{k}}\sigma\otimes e^{j}\\
&\ \ \ +h_{ij}e_{k}e_{l}e_{i}\cdot\nabla^{S}_{e_{k}}\nabla^{S}_{e_{l}}\sigma\otimes e^{j}\\
&=\nabla_{e_{k}}\nabla_{e_{l}}h_{ij}e_{k}e_{l}e_{i}\cdot\sigma\otimes e^{j}
  -2\mu\nabla_{e_{k}}h_{ij}e_{i}e_{k}\cdot\sigma\otimes e^{j}\\
&\ \ \ +\mu^{2}h_{ij}e_{k}e_{l}e_{i}e_{l}e_{k}\cdot\sigma\otimes e^{j}\\
&=-\nabla_{e_{k}}\nabla_{e_{k}}h_{ij}e_{i}\cdot\sigma\otimes e^{j}
  -\frac{1}{2}R_{e_{k}e_{l}}h_{ij}e_{k}e_{l}e_{i}\cdot\sigma\otimes e^{j}\\
&\ \ \ -2\mu\nabla_{e_{k}}h_{ij}e_{i}e_{k}\cdot\sigma\otimes e^{j}
   +(n-2)^{2}\mu^{2}h_{ij}e_{i}\cdot\sigma\otimes e^{j}\\
&=\Phi(\nabla^{*}\nabla h)+\frac{1}{2}R_{kljp}h_{ip}e_{k}e_{l}e_{i}\cdot\sigma\otimes e^{j}
   +\frac{1}{2}R_{klip}h_{pj}e_{k}e_{l}e_{i}\cdot\sigma\otimes e^{j}\\
&\ \ \ -2\mu\nabla_{e_{k}}h_{ij}e_{i}e_{k}\cdot\sigma\otimes e^{j}
   +(n-2)^{2}\mu^{2}\Phi(h).
\end{aligned}
\end{equation}
In the third equality, we use the Clifford relation $e_{k}e_{l}+e_{l}e_{k}=-2\delta_{kl}$, and $\nabla^{S}_{X}\sigma=\mu X\cdot\sigma$ for any vector field $X$. In the fourth equality, we use twice the fact
$$e_{l}e_{i}e_{l}\cdot\phi=(n-2)e_{i}\cdot\phi$$
for any spinor $\phi$, which can easily be obtained by using the Clifford relation.

By using the Clifford relation, $(\ref{CurvatureOnSpinorBundle})$, and $(\ref{RicciCurvatureWithKillingSpinor})$, we have
\begin{equation}\label{TheFirstCurvatureTerm}
\frac{1}{2}R_{kljp}h_{ip}e_{k}e_{l}e_{i}\cdot\sigma\otimes e^{j}=\Phi(-2\mathring{R}h)-4\mu^{2}\Phi(h)+4\mu^{2}trh e_{j}\cdot\sigma\otimes e^{j},
\end{equation}
\begin{equation}\label{TheSecondCurvatureTerm}
\frac{1}{2}R_{klip}h_{pj}e_{k}e_{l}e_{i}\cdot\sigma\otimes e^{j}=4(n-1)\mu^{2}\Phi(h),
\end{equation}
\begin{equation}\label{DivergenceTerm}
-2\mu\nabla_{e_{k}}h_{ij}e_{i}e_{k}\cdot\sigma\otimes e^{j}=-4\mu(\delta h)_{j}\sigma\otimes e^{j}+2\mu e_{k}\cdot\Phi(\nabla_{e_{k}}h),
\end{equation}
\begin{equation}\label{DivergenceTerm2}
e_{k}\cdot\Phi(\nabla_{e_{k}}h)=D\Phi(h)-(n-2)\mu\Phi(h).
\end{equation}
By plugging $(\ref{TheFirstCurvatureTerm}), (\ref{TheSecondCurvatureTerm}), (\ref{DivergenceTerm})$ and $(\ref{DivergenceTerm2})$ into $(\ref{DDPhi(h)})$, we get $(\ref{Bochner})$.
\end{proof}


\section{Stability of Riemannian manifolds with imaginary Killing spinors}
\noindent In this section, we obtain an estimate for the Einstein operator on complete Riemannian manifolds with imaginary Killing spinors of type \uppercase\expandafter{\romannumeral1}. As a consequence of the estimate and Baum's classification results, we prove that all complete Riemannian manifolds with imaginary Killing spinors are strictly stable.

Let $(M^{n},g)$ be a Riemannian manifold with an imaginary Killing spinor $\sigma$ of type \uppercase\expandafter{\romannumeral1} with the Killing constant $\mu=\sqrt{-1}\nu$. We have the following property for the map $\Phi$ defined in $(\ref{DefinitionOfPhi})$.
\begin{lem}\label{ImaginaryInnerProductForPhi}
For all $h, \tilde{h}\in C^{\infty}(M, S^{2}(M))$, we have
\begin{equation}
Re\langle\Phi(h),\Phi(\tilde{h})\rangle=\langle h,\tilde{h}\rangle f,
\end{equation}
where $f=\langle\sigma,\sigma\rangle$ is the length function.
\end{lem}
\begin{proof}
\begin{align*}
Re\langle\Phi(h),\Phi(\tilde{h})\rangle &=Re(h_{ij}\tilde{h}_{kl}\langle e_{i}\cdot\sigma\otimes e^{j}, e_{k}\cdot\sigma\otimes e^{l}\rangle)\\
                                        &=Re(h_{ij}\tilde{h}_{kj}\langle e_{i}\cdot\sigma, e_{k}\cdot\sigma\rangle)\\
                                        &=h_{ij}\tilde{h}_{kj}Re\langle e_{i}\cdot\sigma, e_{k}\cdot\sigma\rangle)\\
                                        &=h_{ij}\tilde{h}_{ij}f.
\end{align*}
In the last step, we use $(\ref{OrthogonalityOfSpinors})$.
\end{proof}
\begin{lem}\label{NormalOftDirectionAction}
If $\sigma$ is a Killing spinor of type \uppercase\expandafter{\romannumeral1} as in $(\ref{KillingSpinor1})$ or $(\ref{KillingSpinor2})$, then we have
\begin{equation}
\|(\frac{\partial}{\partial t})\cdot\Phi(h)\|=\|\Phi(h)\|.
\end{equation}
\end{lem}
\begin{proof}
Choose a local orthonormal frame of TM as $\{e_{1}=\frac{\partial}{\partial r},e_{2},\cdots,e_{n}\}$. Then by $(\ref{tDirectionActionOnKillingSpinor})$, we have
\begin{align*}
(\frac{\partial}{\partial t})\cdot\Phi(h)
&=(\frac{\partial}{\partial t})\cdot(h_{1j}(\frac{\partial}{\partial t})\cdot\sigma\otimes e^{j}+\sum_{i\geq2}h_{ij}e_{i}\cdot\sigma\otimes e^{j})\\
&=\sqrt{-1}h_{1j}(\frac{\partial}{\partial t})\cdot\sigma\otimes e^{j}-\sqrt{-1}\sum_{i\geq2}h_{ij}e_{i}\cdot\sigma\otimes e^{j}.
\end{align*}
Then by $(\ref{OrthogonalityOfSpinors})$, we have
\begin{align*}
\|(\frac{\partial}{\partial t})\cdot\Phi(h)\|^{2}
&=Re\langle (\frac{\partial}{\partial t})\cdot\Phi(h),(\frac{\partial}{\partial t})\cdot\Phi(h)\rangle\\
&=Re\langle \sqrt{-1}h_{1j}(\frac{\partial}{\partial t})\cdot\sigma\otimes e^{j}-\sqrt{-1}\sum_{i\geq2}h_{ij}e_{i}\cdot\sigma\otimes e^{j},\\
&\ \ \ \sqrt{-1}h_{1l}(\frac{\partial}{\partial t})\cdot\sigma\otimes e^{l}-\sqrt{-1}\sum_{k\geq2}h_{kl}e_{k}\cdot\sigma\otimes e^{l}\rangle\\
&=h_{ij}h_{ij}f\\
&=\|\Phi(h)\|^{2}.
\end{align*}
\end{proof}

\begin{thm}\label{ImaginaryKillingEstimate}
Let $(M^{n},g)$ be a complete Riemannian manifold with an imaginary Killing spinor $\sigma$ of type \uppercase\expandafter{\romannumeral1} with Killing constant $\mu=\sqrt{-1}\nu$. Then we have
\begin{equation}
\int_{M}\langle(\nabla^{*}\nabla-2\mathring{R})h, h\rangle dvol_{g}\geq[n(n-2)-4]\nu^{2}\int_{M}\langle h,h\rangle dvol_{g},
\end{equation}
for all compactly supported traceless transverse $h\in C^{\infty}_{0}(M, S^{2}(M))$.
\end{thm}
\begin{proof}
By Proposition \ref{BochnerTypeForumula}, for any compactly supported traceless transverse symmetric 2-tensor $h$,
\begin{equation}\label{Phi(Deltah)}
\Phi((\nabla^{*}\nabla-2\mathring{R})h)=D^{*}D\Phi(h)-n(n-2)\mu^{2}\Phi(h)-2\mu D\Phi(h).
\end{equation}
By Theorem $\ref{ClassificationOfType1}$, we can take a Killing spinor as in $(\ref{KillingSpinor1})$ or $(\ref{KillingSpinor2})$ depending on dimension $n$ of the manifold. Then we know the length function is given by
\begin{equation}\label{lengthfunction}
f=e^{-2\nu t}.
\end{equation}
By $(\ref{Phi(Deltah)})$, and Lemma \ref{ImaginaryInnerProductForPhi}, for all traceless transverse $h\in C^{\infty}_{0}(S^{2}(M))$, we have
\begin{equation}\label{MainEstimate}
\begin{aligned}
\int_{M}\langle(\nabla^{*}\nabla-2\mathring{R})h, h\rangle dvol_{g}
&=\int_{M}\frac{Re\langle\Phi((\nabla^{*}\nabla-2\mathring{R})h),\Phi(h)\rangle}{f} dvol_{g}\\
&=\int_{M}\frac{Re\langle D^{*}D\Phi(h),\Phi(h)\rangle}{f} dvol_{g}\\
&\ \ \ -n(n-2)\mu^{2}\int_{M}\frac{\langle\Phi(h),\Phi(h)\rangle}{f} dvol_{g}\\
&\ \ \ +\int_{M}\frac{Re\langle-2\mu D\Phi(h),\Phi(h)\rangle}{f} dvol_{g}
\end{aligned}
\end{equation}
By using $(\ref{lengthfunction})$ and doing an integration by parts, we obtain
\begin{equation*}
\begin{aligned}
\int_{M}\frac{Re\langle D^{*}D\Phi(h),\Phi(h)\rangle}{f} dvol_{g}
&=\int_{M}\frac{\|D\Phi(h)\|^{2}}{f}dvol_{g}\\
&\ \ \ +\int_{M}\frac{Re\langle D\Phi(h),2\nu(\frac{\partial}{\partial t})\cdot\Phi(h)\rangle}{f}dvol_{g}.
\end{aligned}
\end{equation*}
By Cauchy inequality, we have
\begin{align*}
Re\langle D\Phi(h),2\nu(\frac{\partial}{\partial t})\cdot\Phi(h)\rangle
&\geq -\|D\Phi(h)\|\cdot\|2\nu(\frac{\partial}{\partial t})\cdot\Phi(h)\|\\
&\geq -\frac{\|D\Phi(h)\|^{2}+4\nu^{2}\|(\frac{\partial}{\partial t})\cdot\Phi(h)\|^{2}}{2}\\
&=-\frac{\|D\Phi(h)\|^{2}+4\nu^{2}\|\Phi(h)\|^{2}}{2}
\end{align*}
Thus we have
\begin{equation}\label{EstimateForFirstTerm}
\begin{aligned}
\int_{M}\frac{Re\langle D^{*}D\Phi(h),\Phi(h)\rangle}{f}dvol_{g}
&\geq\frac{1}{2}\int_{M}\frac{\|D\Phi(h)\|^{2}}{f}dvol_{g}\\
&\ \ \ -2\nu^{2}\int_{M}\langle h,h\rangle dvol_{g}.
\end{aligned}
\end{equation}
Similarly, by Cauchy inequality, we have
\begin{equation}\label{EstimateForThirdTerm}
\begin{aligned}
\int_{M}\frac{Re\langle -2\mu D\Phi(h),\Phi(h)\rangle}{f}dvol_{g}
&\geq-\frac{1}{2}\int_{M}\frac{\|D\Phi(h)\|^{2}}{f}dvol_{g}\\
&\ \ \ -2\nu^{2}\int_{M}\langle h,h\rangle dvol_{g}.
\end{aligned}
\end{equation}
By plugging $(\ref{EstimateForFirstTerm})$ and $(\ref{EstimateForThirdTerm})$ into $(\ref{MainEstimate})$, we complete the proof.
\end{proof}

Then Theorem $\ref{ImaginaryKillingEstimate}$ enables us to prove the following stability result recently obtained in \cite{Kro15} in a differential way.
\begin{cor}
Complete Riemannian manifolds with non-zero imaginary Killing spinors are strictly stable.
\end{cor}
\begin{proof}
By Theorem $\ref{ClassificationOfType2}$, complete Riemannian manifolds with Killing spinors of type \uppercase\expandafter{\romannumeral2} are isometric to hyperbolic spaces, and therefore are strictly stable (see \cite{Koi79}, and the proof of Theorem 12.67 in \cite{Bes87}). Let $(M^{n},g)$ be a Riemannian manifold with Killing spinors of type \uppercase\expandafter{\romannumeral1}. If $n\geq4$, then by Theorem $\ref{ImaginaryKillingEstimate}$, $(M^{n},g)$ is strictly stable. If $n\leq3$, we know it has negative constant sectional curvature, and therefore it is also strictly stable.
\end{proof}


\section{Stability of Riemannian manifolds with real Killing spinors}
\noindent In this section, we give a stability condition for manifolds with real Killing spinors in terms of a twisted Dirac operator. Because the length function of a real Killing spinor is constant, an estimate for the Einstein operator can be obtained easier than the case of imaginary Killing spinor. However, unlike imaginary Killing spinor case, from the estimate we cannot conclude a general stability result for manifolds with real Killing spniors.

Let $(M^{n},g)$ be a Riemannian manifold with a real Killing spinor $\sigma$ with Killing constant $\mu$. Without loss of generality, we can choose $\sigma$ to be of unit length.

\begin{lem}\label{RealInnerProductForPhi}
For all $h,\tilde{h}\in C^{\infty}(M, S^{2}(M))$, we have
$$Re\langle\Phi(h),\Phi(\tilde{h})\rangle=\langle h,\tilde{h}\rangle.$$
\end{lem}

Then by Proposition $\ref{BochnerTypeForumula}$, Lemma $\ref{RealInnerProductForPhi}$, and the fact that $\mu\int_{M}\langle D\Phi(h),\Phi(h)\rangle dvol_{g}$ is real, we obtain the following estimate for the Einstein operator $\nabla^{*}\nabla-2\mathring{R}$.
\begin{thm}[\cite{GHP03}, \cite{Wan91}]\label{RealKillingEstimate}
If the Killing constant $\mu$ is real, then, for all traceless transverse $h\in C^{\infty}(M, S^{2}(M))$,
\begin{equation}
\begin{aligned}
\int_{M}\langle(\nabla^{*}\nabla-2\mathring{R})h, h\rangle dvol_{g}
&=\int_{M}\langle D\Phi(h),D\Phi(h)\rangle dvol_{g}\\
&\ \ \ -2\mu\int_{M}\langle D\Phi(h),\Phi(h)\rangle dvol_{g}\\
&\ \ \ -n(n-2)\mu^{2}\int_{M}\langle h,h\rangle dvol_{g}.
\end{aligned}
\end{equation}
\end{thm}
\begin{remark}
As mentioned in \cite{Die13} and \cite{Kro15}, Theorem $\ref{RealKillingEstimate}$ has been used to obtain a lower bound on the eigenvalues of the Einstein operator in \cite{GHP03}. The lower bound is $-(n-1)^{2}\mu^{2}$, as we can also see in the following Corollary $\ref{StabilityWithRealKilling}$.
\end{remark}

\begin{cor}\label{StabilityWithRealKilling}
A Riemannian manifold with a non-zero real Killing spinor with the Killing constant $\mu$ is stable if the twisted Dirac operator $D$ satisfies
$$(D-\mu)^{2}\geq (n-1)^{2}\mu^{2},$$
on $\{\Phi(h):h\in C^{\infty}(M, S^{2}(M)), trh=0, \delta h=0\}$.
\end{cor}
\begin{proof}
By Theorem $\ref{RealKillingEstimate}$, for traceless transverse symmetric 2-tensor $h$, we have
\begin{equation}
\begin{aligned}
\int_{M}\langle(\nabla^{*}\nabla-2\mathring{R})h, h\rangle dvol_{g}=
&\int_{M}\langle (D-\mu)^{2}\Phi(h),\Phi(h)\rangle dvol_{g}\\
&-(n-1)^{2}\mu^{2}\int_{M}\langle h,h\rangle dvol_{g}.
\end{aligned}
\end{equation}
This implies the stability condition.
\end{proof}


\section{Some unstable regular Sasaki-Einstein manifolds}
\noindent In this section, we study instability of regular Sasaki-Einstein manifolds, which would then be essentially total spaces of principal circle bundles over K\"{a}hler-Einstein manifolds with positive first Chern classes. A product of two Einstein manifolds $(B^{n_{1}}, g_{1})$ and $(B^{n_{2}}, g_{2})$ with the same positive Einstein constant is an unstable Einstein manifold. Indeed, $h=\frac{g_{1}}{n_{1}}-\frac{g_{2}}{n_{2}}$ is an unstable traceless transverse direction. We show that if the base manifold of a regular Sasaki-Einstein manifold is a product of two K\"{a}hler-Einstein manifolds then we obtain an unstable direction on the Sasaki-Einstein manifold by lifting this unstable direction on the base K\"{a}hler-Einstein manifold to the total space.

Let us first recall some basic facts about Sasaki manifolds. For details, we refer to \cite{Bla10} and \cite{FOW09}. A quick definition of Sasaki manifolds is given as the following, see, e.g. \cite{FOW09}.
\begin{defn}\label{Definition1OfSasakiManifolds}{\rm (Definition 1 of Sasaki manifolds)}
$(M^{n}, g)$ is said to be a Sasaki manifold if the cone $(\mathbb{R}_{+}\times M, dr^{2}+r^{2}g)$ is K\"{a}hler,
where $\mathbb{R}_{+}=(0, +\infty)$, and $r$ is coordinate on $\mathbb{R}_{+}$.
\end{defn}
\begin{remark}
From Definition $\ref{Definition1OfSasakiManifolds}$, we note that a Sasaki manifold has to be of odd dimension.
\end{remark}
There are several equivalent definitions of Sasaki manifolds. The one given in the following looks more complicated and tells us more about structure on Sasaki manifolds themselves.
\begin{defn}\label{Definition2OfSasakiManifolds}{\rm (Definition 2 of Sasaki manifolds)}
Let $(M^{2p+1}, g, \phi, \eta, \xi)$ be a Riemannian manifold of odd dimension $2p+1$ with a $(1,1)$-tensor $\phi$, 1-form $\eta$, and a vector field $\xi$. It is a Sasaki manifold, if
\begin{enumerate}
\renewcommand{\labelenumi}{(\theenumi)}
\item $\eta\wedge (d\eta)^{p}\neq0,$
\item $\eta(\xi)=1,$
\item $\phi^{2}=-id+\eta\otimes\xi,$
\item $g(\phi X, \phi Y)=g(X, Y)-\eta(X)\eta(Y),$
\item $g(X, \phi Y)=d\eta(X, Y),$
\item the almost complex structure on $M^{2p+1}\times\mathbb{R}$ defined by
      $$J(X,f\frac{d}{dr})=(\phi X-f\xi, \eta(X)\frac{d}{dr})$$
      is integrable,
\end{enumerate}
for all vector fields $X$ and $Y$ on $M^{2p+1}$. The vector $\xi$ is called the Reeb vector field. And this is a regular Sasaki manifold if the Reeb vector field $\xi$ is a regular vector field. If, in addition, $g$ is an Einstein metric, then this is a Sasaki-Einstein manifold.
\end{defn}
\begin{remark}\label{PropertiesOfSasakiStructure}
As consequences of Definition $\ref{Definition2OfSasakiManifolds}$, we have $\phi\xi=0$, $\eta\circ\phi=0$, and $\nabla_{X}\xi=-\phi X$, in particular, $\nabla_{\xi}\xi=0$. Moreover, $\xi$ is a Killing vector field. For details, see, e.g. \cite{Bla10}.
\end{remark}

\begin{remark}
Let us recall one more definition of Sasaki manifold. $(M^{n}, g)$ is a Sasaki manifold if there exists a Killing vector filed $\xi$ of unit length on $M^{n}$ so that the Riemann curvature satisfies the condition
\begin{equation}\label{CurvatureOnSasakiManifolds}
R_{X\xi}Y=-g(\xi,Y)X+g(X,Y)\xi,
\end{equation}
for any pair of vector fields $X$ and $Y$ on $M^{n}$. Then from $(\ref{CurvatureOnSasakiManifolds})$, we can easily see that on a Sasaki-Einstein manifold $(M^{n}, g)$ of dimension $n$, $Ric_{g}=(n-1)g$.
\end{remark}

The relationship between real Killing spinors and the Sasaki-Einstein structures has been observed by Th. Friedrich and I. Kath in \cite{FK89} and \cite{FK90}, and then was further studied by C. B\"{a}r in \cite{Bar93}. We briefly summarize their results as the following.
\begin{thm}[Th. Friedrich and I. Kath, and C. B\"{a}r]\label{KillingSpinorsAndSasakiEinstein}
A complete simply-connected Sasaki-Einstein manifold of dimension $n$ with Einstein constant $n-1$ carries at least 2 linearly independent real Killing spinors with distinct Killing constants equal $\frac{1}{2}$ and $-\frac{1}{2}$ for $n\equiv3 (mod4)$, and to the same Killing number equals $\frac{1}{2}$ for $n\equiv1 (mod4)$, respectively.

Conversely, a complete Riemannian spin manifold with such spinors in these dimensions is Sasaki-Einstein.
\end{thm}

\begin{remark}
Th. Friedrich also proved that a complete 4-dimensional manifold with a real Killing spinor is isometric to the standard sphere in \cite{Fri81}. And O. Hijazi proved the analogous result in dimension 8 in \cite{Hij86}. More generally, C. B\"{a}r proved that all complete manifolds of even dimension $n$, $n\neq6$, with a real Killing spinor are isometric to a standard sphere in \cite{Bar93}. Thus, complete manifolds of even dimension $n$, $n\neq6$, with a real Killing spinor are strictly stable.
\end{remark}

\begin{remark}
In the first part of Theorem $\ref{KillingSpinorsAndSasakiEinstein}$, we need at least two linearly independent real Killing spinors in order to have a Sasaki-Einstein structure. Actually, on a complete Riemannian spin manifold of odd dimension, except 7, existence of one Killing spinor automatically implies the existence of the second one that we need in Theorem $\ref{KillingSpinorsAndSasakiEinstein}$. The 7-dimensional manifolds with a single linearly independent Killing spinor has been studied in \cite{Kat90} and in more details in \cite{FK97}. We also refer to the book \cite{BFGK91}. The Jensen's sphere is a 7-dimensional complete manifold with a single linearly independent Killing spinor, and it is unstable as mentioned in Introduction. We refer to \cite{ADP83}, \cite{Bar93}, \cite{Bes87}, \cite{Jen73}, and \cite{Spa11} for this interesting example.
\end{remark}

Now let us recall the construction of a typical regular Sasaki manifold in \cite{Bla10}. Let $(B^{2p}, G, J)$ be a K\"{a}hler manifold of real dimension $2p$, with the K\"{a}hler form $\Omega=G(\cdot, J\cdot)$, where $G$ is a Riemannian metric and $J$ is an almost complex structure. Then let $\pi: M^{2p+1}\rightarrow B^{2p}$ be a principal $S^{1}$-bundle with a connection $\eta$ with the curvature form $d\eta=2\pi^{*}\Omega$. Let $\xi$ be a vertical vector field on $M^{2p+1}$, generated by $S^{1}$-action, such that $\eta(\xi)=1$, and $\widetilde{X}$ denotes the horizontal lift of $X$ with respect to the connection $\eta$ for a vector field $X$ on $B^{2p}$. We set
\begin{equation}\label{AlmostComplexStructureOnTotalSpace}
\phi X=\widetilde{J\pi_{*}X},
\end{equation}
and
\begin{equation}\label{RiemannianMetricOnTotalSpace}
g(X, Y)=G(\pi_{*}X, \pi_{*}Y)+\eta(X)\eta(Y),
\end{equation}
for vector fields $X$ and $Y$ on $M^{2p+1}$. Then $(M^{2p+1}, g, \phi, \eta, \xi)$ is a regular Sasaki manifold.

Conversely, any regular Sasaki manifold can be obtained in this way, see, e.g. Theorem 3.9 and Example 6.7.2 in \cite{Bla10}. Moreover, if $(M^{2p+1}, g)$ is Sasaki-Einstein with Einstein constant $2p$, then $(B^{2p}, G, J)$ is K\"{a}hler-Einstein with Einstein constant $2p+2$.

We fix some notations before carrying on calculations. $\nabla^{g}$ and $\nabla^{G}$ denote the Levi-Civita connections on $(M^{2p+1}, g)$ and on $(B^{2p}, G)$, respectively. $R^{g}$ and $Ric^{g}$, and $R^{G}$ and $Ric^{G}$ denote Riemann and Ricci curvatures on $(M^{2p+1}, g)$ and on $(B^{2p}, G)$, respectively. In the rest of this section, we use $X, Y, Z, W, \cdots$ to denote vector fields on $B^{2p}$, and we use $\widetilde{X}, \widetilde{Y}, \widetilde{Z}, \widetilde{W}, \cdots$ to denote their horizontal lift to $M^{2p+1}$ with respect to the connection $\eta$. And we choose and fix a local orthrnormal frame $\{X_{1}, X_{2}, \cdots, X_{2p}\}$ of $TB$. Then $\{\widetilde{X_{1}}, \widetilde{X_{2}}, \cdots, \widetilde{X_{2p}}, \xi\}$ is a local orthonormal frame of $TM$. We use $\nabla^{g}_{i}$ to denote $\nabla^{g}_{\widetilde{X_{i}}}$, and $\nabla^{G}_{i}$ to denote $\nabla^{G}_{X_{i}}$.

\begin{lem}\label{RelationForCovariantDerivative}
On a regular Sasaki manifold $(M^{2p+1}, g, \phi, \eta, \xi)$ constructed above. We have
\begin{equation}\label{LieDerivative}
[\xi, \widetilde{X}]=\mathcal{L}_{\xi}\widetilde{X}=0,
\end{equation}
\begin{equation}\label{CovariantDerivativeForHorizontalVectors}
\nabla^{g}_{\widetilde{X}}\widetilde{Y}=\widetilde{\nabla^{G}_{X}Y}-\Omega(X, Y)\xi,
\end{equation}
\begin{equation}\label{CovariantDerivativeForVerticalHorizontalVectors}
\nabla^{g}_{\xi}\widetilde{X}=\nabla^{g}_{\widetilde{X}}\xi=-\phi\widetilde{X},
\end{equation}
\begin{equation}\label{CovariantDerivativeForVerticalVectors}
\nabla^{g}_{\xi}\xi=0.
\end{equation}
\end{lem}
\begin{proof}
The equality $(\ref{LieDerivative})$ follows from the fact that the horizontal distribution is $S^{1}$ invariant and $\xi$ is generated by the $S^{1}$-action. Then the rest properties for covariant derivatives follow from properties in Remark $\ref{PropertiesOfSasakiStructure}$, $(\ref{LieDerivative})$, and the fundamental equations of a submersion in \cite{One66} (also see \cite{Bes87} for the equations).
\end{proof}

Let $h\in C^{\infty}(B, S^{2}(B))$, and then $\tilde{h}=\pi^{*}h\in C^{\infty}(M, S^{2}(M))$. Then by Lemma $\ref{RelationForCovariantDerivative}$ and straightforward calculations, we obtain a relationship between $(\nabla^{g})^{*}\nabla^{g}\widetilde{h}$ and $(\nabla^{G})^{*}\nabla^{G}h$.
\begin{lem}\label{RelationBetweenLaplacian}
\begin{equation}
(\nabla^{g}_{k}\nabla^{g}_{k}\tilde{h})_{ij}=(\pi^{*}(\nabla^{G}_{k}\nabla^{G}_{k}h))_{ij}-2\tilde{h}_{ij},
\end{equation}
\begin{equation}
(\nabla^{g}_{\nabla^{g}_{k}\widetilde{X_{k}}}\tilde{h})_{ij}=(\pi^{*}(\nabla^{G}_{\nabla^{G}_{k}X_{k}}h))_{ij},
\end{equation}
\begin{equation}
(\nabla^{g}_{\xi}\nabla^{g}_{\xi}\tilde{h})_{ij}=-2\tilde{h}_{ij}+2\tilde{h}(\phi \widetilde{X_{i}}, \phi \widetilde{X_{j}}),
\end{equation}
and therefore,
\begin{equation}
((\nabla^{g})^{*}\nabla^{g}\widetilde{h})_{ij}=(\pi^{*}((\nabla^{G})^{*}\nabla^{G}h))_{ij}+4\tilde{h}_{ij}-2\tilde{h}(\phi \widetilde{X_{i}}, \phi \widetilde{X_{j}}),
\end{equation}
for all $1\leq i, j\leq 2p$, where we take summation for the repeated index $k$ through $1$ to $2p$.
\end{lem}

Because $\pi: M^{2p+1}\rightarrow B^{2p}$ is a Riemannian submersion, by the fundamental equation in \cite{One66} and also in Theorem 9.26 in \cite{Bes87}, we have the following relationship between curvature tensors on $M^{2p+1}$ and ones on $B^{2p}$.
\begin{lem}
\begin{equation}
\begin{aligned}
R^{g}(\widetilde{X}, \widetilde{Y}, \widetilde{Z}, \widetilde{W})
&=(\pi^{*}R^{G})(X, Y, Z, W)\\
&\ \ \ -2(\pi^{*}\Omega)(\widetilde{X}, \widetilde{Y})(\pi^{*}\Omega)(\widetilde{Z}, \widetilde{W})\\
&\ \ \ -(\pi^{*}\Omega)(\widetilde{X}, \widetilde{Z})(\pi^{*}\Omega)(\widetilde{Y}, \widetilde{W})\\
&\ \ \ +(\pi^{*}\Omega)(\widetilde{X}, \widetilde{W})(\pi^{*}\Omega)(\widetilde{Y}, \widetilde{Z}),
\end{aligned}
\end{equation}
\begin{equation}
R^{g}(\widetilde{X}, \xi, \widetilde{Y}, \xi)=g(\widetilde{X}, \widetilde{Y}),
\end{equation}
and therefore,
\begin{equation}\label{RelationBetweenRicciCurvature}
Ric^{g}(\widetilde{X}, \widetilde{Y})=(\pi^{*}Ric^{G})(\widetilde{X}, \widetilde{Y})-2g(\widetilde{X}, \widetilde{Y}).
\end{equation}
\end{lem}

From $(\ref{RelationBetweenRicciCurvature})$, we can see that if $g$ is Einstein with Einstein constant $k$ then $G$ is also Einstein with Einstein constant $k+2$. Moreover, the above relations between curvatures directly imply a relation between $\mathring{R}^{g}\tilde{h}$ and $\mathring{R}^{G}h$.
\begin{lem}\label{RelationBetweenCurvatureContraction}
\begin{equation}
(\mathring{R}^{g}\tilde{h})_{ij}=(\pi^{*}(\mathring{R}^{G}h))_{ij}-3\tilde{h}(\phi\widetilde{X_{i}}, \phi\widetilde{X_{j}})-(\pi^{*}\Omega)(\widetilde{X_{i}}, \widetilde{X_{j}})\sum^{2p}_{k=1}\tilde{h}(\widetilde{X_{k}}, \phi\widetilde{X_{k}}),
\end{equation}
for all $1\leq i, j\leq 2p$.
\end{lem}

\begin{thm}\label{EinsteinOperatorRelation}
\begin{equation}
\langle(\nabla^{g})^{*}\nabla^{g}\tilde{h}-2\mathring{R}^{g}\tilde{h}, \tilde{h}\rangle=(\langle(\nabla^{G})^{*}\nabla^{G}h-2\mathring{R}^{G}h, h\rangle+4\langle h, h\rangle+4\langle h\circ J, h\rangle)\circ\pi.
\end{equation}
Therefore,
\begin{equation}
\begin{aligned}
&\ \ \ \int_{M}\langle(\nabla^{g})^{*}\nabla^{g}\tilde{h}-2\mathring{R}^{g}\tilde{h}, \tilde{h}\rangle dvol_{g}\\
&=\int_{B}(\langle(\nabla^{G})^{*}\nabla^{G}h-2\mathring{R}^{G}h, h\rangle+4\langle h, h\rangle+4\langle h\circ J, h\rangle)dvol_{G}.
\end{aligned}
\end{equation}
\end{thm}
\begin{proof}
By Lemma $\ref{RelationBetweenLaplacian}$ and Lemma $\ref{RelationBetweenCurvatureContraction}$, we directly have
\begin{equation}
\begin{aligned}
\langle(\nabla^{g})^{*}\nabla^{g}\tilde{h}-2\mathring{R}^{g}\tilde{h}, \tilde{h}\rangle=
&(\langle(\nabla^{G})^{*}\nabla^{G}h-2\mathring{R}^{G}h, h\rangle+4\langle h, h\rangle\\
&+4\langle h(J\cdot, J\cdot), h\rangle+2(tr_{G}(h(J\cdot, \cdot)))^{2})\circ\pi.
\end{aligned}
\end{equation}
Then it suffices to show that $tr_{G}(h(J\cdot, \cdot)=0$. Because $(B^{2p}, G, J)$ is K\"{a}hler, and in particular complex, we can choose a local orthonormal frame of $TB$ in the form of $\{X_{1}, \cdots, X_{p}, JX_{1}, \cdots, JX_{p}\}$. Then
$$tr_{G}(h(J\cdot, \cdot))=\sum^{p}_{i=1}h(JX_{i}, X_{i})+\sum^{p}_{j=1}h(J^{2}X_{j}, JX_{j})=0,$$
by using $J^{2}=-id$ and the symmetry of $h$.
\end{proof}

We choose a local orthonormal frame $\{X_{1}, \cdots, X_{p}, JX_{1}, \cdots, JX_{p}\}$ of $TB$ as in the proof of Proposition $\ref{EinsteinOperatorRelation}$, and set
$$h(X_{i}, X_{j})=h_{ij}, \ \ h(X_{i}, JX_{j})=h_{i\bar{j}}, \ \ h(JX_{i}, X_{j})=h_{\bar{i}j}, \ \ h(JX_{i}, JX_{j})=h_{\bar{i}\bar{j}},$$
for all $1\leq i,j\leq p$.

Then we have
\begin{equation}
\langle h, h\rangle=\sum^{p}_{i,j=1}(h_{ij}h_{ij}+h_{i\bar{j}}h_{i\bar{j}}+h_{\bar{i}j}h_{\bar{i}j}+h_{\bar{i}\bar{j}}h_{\bar{i}\bar{j}}),
\end{equation}
\begin{equation}\label{InequalityFor2-tensor}
\langle h\circ J, h\rangle=\sum^{p}_{i,j=1}2(h_{ij}h_{\bar{i}\bar{j}}-h_{\bar{i}j}h_{i\bar{j}})\leq\langle h, h\rangle.
\end{equation}

For any $h\in C^{\infty}(S^{2}(B))$, by doing directly calculations, we have that $tr_{g}\tilde{h}=tr_{G}h$, $(\delta_{g}\tilde{h})(\widetilde{X})=(\delta_{G}h)(X)$, and $(\delta_{g}\tilde{h})(\xi)=-tr_{G}(h(J\cdot, \cdot))=0$. Consequently, if $h$ is traceless and transverse, then so is $\tilde{h}$.
\begin{cor}
If there exists a traceless transverse symmetric 2-tensor $h\in C^{\infty}(B, S^{2}(B))$ such that $\int_{B}\langle(\nabla^{G})^{*}\nabla^{G}h-2\mathring{R}^{G}h, h\rangle dvol_{G}\leq-8\int_{B}\langle h, h\rangle dvol_{G}$, then $(M^{2p+1}, g)$ is unstable.
\end{cor}
\begin{proof}
Proposition $\ref{EinsteinOperatorRelation}$ and the inequality $(\ref{InequalityFor2-tensor})$ directly imply the conclusion.
\end{proof}
\begin{cor}\label{UnstabilityForProductBase}
If the base space $(B^{2p}, G)$ of a regular Sasaki-Einstein manifold $(M^{2p+1}, g)$ is the Riemannian product of K\"{a}hler-Einstein manifolds $(B^{2p_{1}}_{1}, G_{1})$ and $(B^{2p_{2}}_{2}, G_{2})$, where $p_{1}+p_{2}=p$, then $(M^{2p+1}, g)$ is unstable.
\end{cor}
\begin{proof}
Set $h=\frac{G_{1}}{2p_{1}}-\frac{G_{2}}{2p_{2}}$. $h$ is a traceless transverse symmetric 2-tensor and is an unstable direction of $(B^{2p}, G)=(B^{2p_{1}}_{1}, G_{1})\times (B^{2p_{2}}_{2}, G_{2})$. Let us recall
\begin{equation}
Ric_{g}=(2p_{1}+2p_{2})g,
\end{equation}
\begin{equation}
Ric_{G}=(2p_{1}+2p_{2}+2)G.
\end{equation}
Then we have
\begin{equation}
\langle(\nabla^{G})^{*}\nabla^{G}h-2\mathring{R}^{G}h, h\rangle=-2\frac{R_{G_{1}}}{4p^{2}_{1}}-2\frac{R_{G_{2}}}{4p^{2}_{2}}=-2(p_{1}+p_{2}+1)(\frac{1}{p_{1}}+\frac{1}{p_{2}}).
\end{equation}
Moreover,
\begin{equation}
\langle h, h\rangle=\langle h\circ J, h\rangle=\frac{1}{2p_{1}}+\frac{1}{2p_{2}}.
\end{equation}
Thus, by Theorem $\ref{EinsteinOperatorRelation}$, we have
\begin{equation}
\langle(\nabla^{g})^{*}\nabla^{g}\tilde{h}-2\mathring{R}^{g}\tilde{h}, \tilde{h}\rangle=-2(p_{1}+p_{2}-1)(\frac{1}{p_{1}}+\frac{1}{p_{2}})<0,
\end{equation}
if both $p_{1}\geq1$ and $p_{2}\geq1$.
\end{proof}

\section{Acknowlegments}

The author would like to thank his advisor Professor Xianzhe Dai, and also Professor Guofang Wei for introducing the problem to him, and for their careful guidance, valuable discussions and constant encouragement. He thanks Klaus Kr\"{o}ncke for helpful discussions. He also thanks Chenxu He for his interests.



\end{document}